\documentclass[11pt]{amsart}

\usepackage{amssymb}
\usepackage{amsmath}
\usepackage{amsthm}
\usepackage{enumerate}
\usepackage{graphicx}
\usepackage{blkarray}
\usepackage{mathrsfs}

\usepackage{caption}
\usepackage{subcaption}

\usepackage[hidelinks]{hyperref}

\theoremstyle{plain}
\newtheorem{thm}{Theorem}[section]

\newtheorem{lem}[thm]{Lemma}

\numberwithin{equation}{section}

\theoremstyle{definition}

\theoremstyle{remark}
\newtheorem{remark}[thm]{Remark}

\newtheorem*{acknowledgements}{Acknowledgements}

\theoremstyle{plain}

\newcommand{\thmref}[1]{Theorem~\ref{#1}}

\newcommand{\lemref}[1]{Lemma~\ref{#1}}

\newcommand{\calT}{{\mathcal T}}

\newcommand{\calV}{{\mathcal V}}
\newcommand{\calW}{{\mathcal W}}

\newcommand{\HH}{{\mathbb H}}
\newcommand{\RR}{{\mathbb R}}

\newcommand{\ZZ}{{\mathbb Z}}

\newcommand{\eps}{{\varepsilon}}

\DeclareMathOperator{\SL}{SL}

\DeclareMathOperator{\Out}{Out}
\DeclareMathOperator{\Mod}{Mod}
\DeclareMathOperator{\PMod}{PMod}

\DeclareMathOperator{\Fix}{Fix}
\DeclareMathOperator{\Aut}{Aut}
\DeclareMathOperator{\Isom}{Isom}

\newcommand{\id}{\mathrm{id}}

\begin{document}

\title[Failure of the well-rounded retract]{Failure of the well-rounded retract for Outer space and Teichm\"uller space}
\author{Maxime Fortier Bourque}
\address{D\'epartement de math\'ematiques et de statistique, Universit\'e de Montr\'eal, 2920, chemin de la Tour, Montr\'eal (QC), H3T 1J4, Canada}
\email{maxime.fortier.bourque@umontreal.ca}

\begin{abstract}
The well-rounded retract for $\SL_n(\ZZ)$ is defined as the set of flat tori of unit volume and dimension $n$ whose systoles generate a finite-index subgroup in homology. This set forms an equivariant spine of minimal dimension for the space of flat tori.

For both the Outer space $X_n$ of metric graphs of rank $n$ and the Teichm\"uller space $\calT_g$ of closed hyperbolic surfaces of genus $g$, we show that the literal analogue of the well-rounded retract does not contain an equivariant spine. We also prove that the sets of graphs whose systoles fill either topologically or geometrically (two analogues of a set proposed as a spine for $\calT_g$ by Thurston) are spines for $X_n$ but that their dimension is larger than the virtual cohomological dimension of $\Out(F_n)$ in general.
\end{abstract}

\maketitle

\section{Introduction}

A \emph{systole} in a compact metric space is a non-contractible closed curve of  minimal length among such curves. 
Ash \cite{Ash} defined the \emph{well-rounded retract} $\mathfrak{W}_n$ of the space $\mathfrak{T}_n$ of marked flat tori of unit volume and dimension $n$ as the set of tori $T$ whose systoles generate a finite-index subgroup in $H_1(T,\ZZ)$, following work of Soulé \cite{Soule} in the case $n=3$. He proved that there is an $\SL_n(\ZZ)$-equivariant deformation retraction of $\mathfrak{T}_n$ onto $\mathfrak{W}_n$, i.e., that $\mathfrak{W}_n$ is an equivariant \emph{spine} for $\mathfrak{T}_n$. Furthermore, the quotient $\mathfrak{W}_n / \SL_n(\ZZ)$ is compact and the dimension of $\mathfrak{W}_n$ is equal to the virtual cohomological dimension (vcd) of $\SL_n(\ZZ)$, the smallest possible for a spine.

Motivated by this, Thurston \cite{Thurston} considered the set $\calV_g$ of marked closed hyperbolic surfaces of genus $g$ whose systoles \emph{fill}, meaning that each component of the complement of their union is contractible (hence the interior of a polygon). Equivalently, a set $C$ of curves on a surface $S$ fills if every non-contractible closed curve in $S$ intersects some element of $C$. Thurston sketched a proof that there is a mapping class group equivariant deformation retract of the Teichm\"uller space $\calT_g$ onto $\calV_g$, but his argument had gaps \cite{Ji}. Furthermore, the dimension of $\calV_g$ is larger than the vcd of the mapping class group $\Mod_g$ in general \cite{dim}.

The third character in this story is the Culler--Vogtmann Outer space $X_n$ of marked metric graphs of unit volume and rank equal to $n$, upon which the group $\Out(F_n)$ of outer automorphisms of the free group of rank $n$ acts. Culler and Vogtmann \cite{CullerVogtmann} found a cocompact equivariant spine $K_n$ for $X_n$ of dimension $2n-3$, equal to the vcd of $\Out(F_n)$ (which was determined using $K_n$). This spine $K_n$ is not defined in terms of systoles.

If $(E,G)$ is equal to either $(\mathfrak{T}_n,\SL_n(\ZZ))$, $(\calT_g,\Mod_g)$, or $(X_n,\Out(F_n))$, and $x \in E$, then
\[
\dim(E) = \mathrm{vcd}(G) + \mathrm{rank}(H_1(x,\ZZ))-1,
\]
which suggests that one should use homology to define spines of minimal dimension. A naive approach is to simply transpose the definition of the well-rounded retract $\mathfrak{W}_n$ in the other two settings. That is, we can define the set $W_n \subset X_n$ of graphs whose systoles generate a finite-index subgroup in integral homology and the set $\calW_g \subset \calT_g$ of hyperbolic surfaces whose systoles generate a finite-index subgroup in integral homology. In his PhD thesis \cite{Baker}, Baker proved that $W_3$ is a spine of minimal dimension for $X_3$ different from $K_3$. However, these analogues $W_n$ and $\calW_g$ of the well-rounded retract fail to achieve their goal in general.

\begin{thm} \label{thm:not_spine_outer}
There exist infinitely many $n\geq 2$ such that $W_n$ does not contain any $\mathrm{Out}(F_n)$-equivariant spine for $X_n$. 
\end{thm}

\begin{thm} \label{thm:not_spine_teich}
There exist infinitely many $g\geq 2$ such that $\calW_g$ does not contain any $\Mod_g$-equivariant spine for $\calT_g$.
\end{thm}

Note that the dimension of $W_n$ (resp. $\calW_g$) is equal to the vcd of $\Out(F_n)$ (resp. $\Mod_g$). The obstruction comes instead from the fact that these sets miss certain loci of fixed points of finite subgroups that have to intersect any spine.

These theorems go in the same direction as results of Pettet and Souto showing that $\mathfrak{W}_n$ is a minimal spine \cite{PettetSouto2} and slightly modifying its de\-fi\-ni\-tion can yield sets of the same dimension that are not spines anymore \cite{PettetSouto1}. In other words, spines are sensitive and thus tricky to find.

There is also an analogue of the Thurston set $\calV_g$ in $X_n$. Indeed, consider the set $V_n \subset X_n$ of graphs whose systoles \emph{topologically fill}, meaning that each component of the complement of their union is contractible. Equivalently, a set $C$ of closed geodesics in a metric graph $\Gamma$ topologically fills if every non-contractible curve in $\Gamma$ intersects some element of $C$. One could also consider the set $V_n'$ of graphs whose systoles \emph{geometrically fill} in the sense that their union is equal to the whole graph. It is easy to see that
\[
W_n \subseteq V_n \quad \text{and} \quad   V_n' \subseteq V_n.
\]
Furthermore, $V_2'$ coincides with $K_2$, the dual to the Farey triangulation, but $W_2=V_2$ is strictly larger (it contains the dumbbells with two loops of equal length, which form spikes emanating from the midpoints of the edges in $K_2$).

In contrast with $W_n$, the sets $V_n$ and $V_n'$ are always spines.

\begin{thm}\label{thm:spine}
For every $n \geq 2$, the set $V_n$ is an $\Out(F_n)$-equivariant spine for $X_n$ and $V_n'$ is an equivariant spine for $V_n$. 
\end{thm}

However, their dimension is too large in general.

\begin{thm}\label{thm:high_dimension}
For every $\eps\in (0,1)$, there exists an $n$ such that the dimension of $V_n'$ is at least $(3-\eps)n$, hence larger than the vcd of $\Out(F_n)$.
\end{thm}

Note that there is also a dynamically-defined notion of filling currents for free groups due to Kapovich and Lustig \cite{filling}. We do not know whether the set of graphs whose systoles fill in that sense forms a spine and if so, what its dimension is.

One may wonder if there is a spine of the minimal dimension $2n-3$ contained in $V_n'$. However, it seems difficult to push the deformation retraction defined in the proof of \thmref{thm:spine} much further. One can continue until there is a systole passing through any pair of edges that are adjacent at a vertex of degree $3$ by folding these edges gradually otherwise, but the proof of \thmref{thm:high_dimension} implies that the  dimension of the resulting set is still too large in general.

\begin{acknowledgements}
I thank Thomas Haettel for making a remark which sparked this project after a talk I gave at UQAM, Karen Vogtmann for pointing out the reference \cite{Baker}, and the referee for useful comments and corrections.
\end{acknowledgements}

\section{Outer space}

We start by proving the negative results regarding Outer space. The proof of both \thmref{thm:not_spine_outer} and \thmref{thm:high_dimension} is based on the same family of graphs that have a large automorphism group and few systoles that cover the whole graph. These graphs were used in \cite{sublinear} to construct hyperbolic surfaces with similar properties.

Given integers $p,q \geq 2$, a \emph{map of type $\{p,q\}$} is a connected graph of constant valence (degree) $q$ embedded in an oriented surface such that each complementary region (whose closure is called a \emph{face}) is a topological disk whose boundary consists of $p$ edges. This can also be phrased in terms of a ribbon structure on the graph. A \emph{flag} is a triple $(v,e,f)$ where $v$ is a vertex, $e$ is an edge, $f$ is a face, and $v \subset e \subset f$. A map is \emph{flag-transitive} if for any two flags there is a homeomorphism of the underlying surface which sends the map to itself and the first flag to the second. For now we consider our maps as combinatorial graphs where each edge has length $1$. The \emph{girth} of a combinatorial graph is the same as its systole, namely, the minimal length of a cycle that is not contractible.

We will require a small variation of a result of Evans \cite[Theorem 11]{Evans} about the existence of flag-transitive maps of large girth. The difference here is that we want to make sure that only the obvious cycles have length equal to the girth.

\begin{lem} \label{lem:evans}
For any $q\geq 3$ and $p\geq 7$, there exists a finite flag-transitive map $M$ of type $\{p,q\}$ and girth $p$ such that the only non-trivial cycles of length $p$ in $M$ are the face boundaries.
\end{lem}

\begin{proof}
There is an infinite flag-transitive map $M_{p,q}$ of type $\{p,q\}$ embedded in the hyperbolic plane $\HH^2$ coming from the tiling by regular $p$-gons with interior angles $2\pi / q$. The automorphisms of $M_{p,q}$ are realized by a finitely-generated discrete group $G$ of isometries of the hyperbolic plane. By Mal’cev’s theorem \cite{Malcev}, $G$ is residually finite, so there is a sequence of normal subgroups $G_k \triangleleft G$ of finite index such that $\bigcap G_k = \{ \id \}$. This implies that $G_k$ is eventually torsion-free and the closed hyperbolic surfaces $S_k =\HH^2 / G_k$ have injectivity radius going to infinity as $k \to \infty$. If $k$ is large enough, then the projection $M_k$ of $M_{p,q}$ to $S_k$ has type $\{p,q\}$ because the map $\HH^2 \to S_k$ is a covering map. Furthermore, $M_k$ is finite since $M_{p,q}/G$ is a half-edge and $G_k$ has finite index in $G$. Lastly, $M_k$ is flag-transitive via the quotient group $G/G_k$ acting on $S_k$. 

Since the face boundaries in $M_k$ have combinatorial length $p$, the girth of $M_k$ is at most $p$.
Since the injectivity radius of $S_k$ tends to infinity, any cycle in $M_k$ which is not contractible in $S_k$ becomes arbitrarily long (with respect to the hyperbolic metric and therefore also in terms of its number of edges) as $k$ tends to infinity. In particular, a cycle in $M_k$ that is not contractible in $S_k$ has combinatorial length strictly larger than $p$ if $k$ is large enough. It is also true that any cycle in $M_k$ which is contractible in $S_k$ (and hence lifts to the universal cover) has combinatorial length at least $p$ with equality only if it is the boundary of a face. We can prove this as follows. Suppose that $\gamma$ is an embedded cycle of combinatorial length at most $p$ in $M_{p,q}$. Let $A$ be the hyperbolic area of any face in the tiling, let $N$ be the number of faces enclosed by $\gamma$, and for a vertex $v \in \gamma$ let $k_v$ be the geodesic curvature of $\gamma$ at $v$, that is, $\pi$ minus the interior angle. Then the Gauss--Bonnet formula yields
\[
2 \pi = \sum_{v \in \gamma} k_v - N\cdot A \leq p \left(1- \frac{2}{q}\right) \pi - A = 2\pi
\]
so that in fact $N=1$ and $\gamma$ is the boundary of a face.
\end{proof}

To prove our results, we use this construction with $q=3$ and $p \geq 7$ arbitrarily large. Let $M$ be a map satisfying the conclusions of \lemref{lem:evans} with these parameters and let $V$, $E$, and $F$ be its number of vertices, edges, and faces respectively. Then 
\[
3 V = 2 E = p F.
\] 
The rank $n$ of $M$ is such that its Euler characteristic is \[1-n = V - E = -V/2\] so that $n = 1+V/2$. By the lemma, the systoles in $M$ are the faces boun\-da\-ries, so there are $F = 3V/p = \frac{6}{p}(n-1)$ of them. In particular, the number of systoles divided  by the rank $n$ is arbitrarily small if $p$ is large enough.

\begin{proof}[Proof of \thmref{thm:high_dimension}]
Given $\eps \in (0,1)$, choose $p\geq 7$ such that $6/p < \eps$, then let $M$ be a finite map of type $\{p,3\}$ as above whose combinatorial systoles are the face boundaries.

Let $n$ be the rank of $M$, pick an arbitrary homotopy equivalence from the bouquet on $n$ circles to $M$ to get a marking, and make all edges of $M$ of equal length $1/E$ so that its volume is $1$. We can now consider $M$ as an element in the Outer space $X_n$. Since the systoles in $M$ are the face boundaries, they cover the whole graph so that $M\in V_n'$.

We now want to deform $M$ (i.e., vary the lengths on its edges) in such a way that the systoles stay the same curves and thus still cover the whole graph. Since competing curves are longer by a definite amount, near $M$ these curves will remain systoles as long as they stay of equal length. 

If $\gamma_1, \ldots, \gamma_F$ are the systoles, then this requires $F-1$ equations, namely, 
\[
\ell(\gamma_1)= \ell(\gamma_2), \quad  \ell(\gamma_2) = \ell(\gamma_3), \quad \ldots, \quad  \ell(\gamma_{F-1}) = \ell(\gamma_F).
\]
In turn, each $\ell(\gamma_j)$ is equal to the sum of the lengths of the edges traversed by $\gamma_j$, so this gives us $F-1$ linear equations for the edge lengths. The subspace of $\RR^E$ cut out by these equations has codimension at most $F-1$ and then we intersect this with the hyperplane where the sum of the lengths is equal to $1$. The dimension of the intersection $I$ is at least 
\[
E-F = \left(3-\frac{6}{p}\right)(n-1)
\]
and this is larger than $(3 - \eps)n$ provided that $p$ (and hence $n$) is large enough. 

As explained above, there is a neighborhood $U$ of $M$ in $I$ where the face boundaries will remain systoles and hence $U \subset V_n'$. This shows that the dimension of $V_n'$ is at least $(3 - \eps)n$. Since $\eps<1$, this is strictly larger than $2n-3$, the vcd of $\Out(F_n)$. 
\end{proof}

To prove that the well-rounded set $W_n$ does not contain an equivariant spine, we will use the above construction together with the following elementary observation, in which $\Fix(H)$ denotes the set of all points fixed by all the elements in $H$.

\begin{lem} \label{lem:fix}
Let $G$ be a group acting on a topological space $E$ and let $S \subseteq E$ be a $G$-equivariant spine for $E$. If $H$ is a subgroup of $G$, then $S \cap\Fix(H)$ is a $G$-equivariant spine for $\Fix(H)$. In particular, if $\Fix(H)\neq \varnothing$ then $S \cap\Fix(H)\neq \varnothing$.
\end{lem}
\begin{proof}
Let $(x,t) \mapsto f_t(x)$ be a continuous map from $E \times [0,1]$ to $E$ such that $f_0$ is the identity on $E$, $f_1(E) = S$, $f_t(x)=x$ for every $x \in S$ and every $t \in [0,1]$, and $f_t(g(x))=g(f_t(x))$ for every $g\in G$, every $x \in X$, and every $t \in [0,1]$. Then for every $t \in [0,1]$, every $h \in H$, and every $x \in \Fix(H)$, we have \[h(f_t(x)) = f_t(h(x))=f_t(x)\] so that $f_t(x) \in \Fix(H)$. This shows that $f_t$ restricts to a map from $\Fix(H)$ to $\Fix(H)$ for all $t \in [0,1]$. This restriction is still $G$-equivariant and equal to the identity on $S \cap \Fix(H)$. Finally, we have \[f_1(\Fix(H)) \subseteq f_1(E) \cap \Fix(H) = S \cap \Fix(H) = f_1(S \cap \Fix(H)) \subseteq f_1(\Fix(H)) \] so that $f_1(\Fix(H)) = S \cap \Fix(H)$. In particular, if $\Fix(H)\neq \varnothing$ then \[S \cap \Fix(H)=f_1(\Fix(H)) \neq \varnothing. \qedhere\]
\end{proof}

We can now prove that $W_n$ does not contain any spine.

\begin{proof}[Proof of \thmref{thm:not_spine_outer}]
Take any $p\geq 7$ and let $M$ be a finite flag-transitive map of type $\{3, p \}$ such that its systoles are the face boundaries. Recall that there are $\frac{6}{p}(n-1)< (n-1)$ systoles in $M$ where $n$ is the rank. In particular, the systoles in $M$ do not generate a finite-index subgroup in $H_1(M,\ZZ) \cong \ZZ^n$. Considering $M$ as a point in $X_n$ after taking a marking and rescaling the metric, this means that $M \notin W_n$.

On the other hand, the stabilizer $H$ of $M$ in $\Out(F_n)$ is isomorphic to the automorphism group of $M$ via the homotopy equivalences between the bouquet on $n$ circles and $M$. Since the quotient $M/\Aut(M)$ is a half-edge whose deformation space is a point, $M$ is the unique fixed point of the group $H$. If there is an equivariant spine $S$ contained in $W_n$, then we have $M \in S \subseteq W_n$ by \lemref{lem:fix}, contradicting $M \notin W_n$. We conclude that $W_n$ does not contain a spine.
\end{proof}

\begin{remark}
By taking $p$ sufficiently large in the above proof, we see that for any $\rho>0$, there exists some $n$ such that the set of graphs in $X_n$ that have at least $\rho\,  n$ systoles does not contain an equivariant spine.
\end{remark}

We end this section by proving the positive result that the sets $V_n$ and $V_n'$ are spines for $X_n$. The proof is the same as for Ash's well-rounded retract $\mathfrak{W}_n$. Recall that $V_n$ is the set of metric graphs (of volume $1$ and rank $n$) whose systoles are such that every component of the complement of their union is contractible (hence a finite tree without its leaves) and $V_n'$ is the set of graphs whose systoles cover the whole graph.  

\begin{proof}[Proof of \thmref{thm:spine}]
The deformation retract is performed in stages. Let $U_k$ be the set of metric graphs such that the union of the systoles is a (possibly disconnected) graph with first Betti number at least $k$. Note that $U_1 = X_n$, $U_n = V_n$, and $U_k \supset U_{k+1}$ for every $k$. It thus suffices to construct an equivariant deformation retraction of $U_k$ onto $U_{k+1}$ for every $k$ and then a deformation retraction of $V_n$ onto $V_n'$.

Let $\Gamma \in U_k\setminus U_{k+1}$ for some $k\geq 1$, let $S_\Gamma$ be the set of edges that belong to some systole and let $T_\Gamma$ be the set of remaining edges. Let $s=s(\Gamma)<1$ be the total length of $S_\Gamma$ so that the total length of $T_\Gamma$ is $1-s>0$. For $0\leq t \leq \log(1/s)$, we define $\Gamma_t$ by rescaling the edges in $S_\Gamma$ by a factor of $e^t$ and those in $T_\Gamma$ by a factor of $\frac{1-e^{t}s}{1-s} \geq 0$ so that the volume remains equal to $1$. For $t$ sufficiently small, the set of systoles in $\Gamma_t$ stays constant because the next shortest closed geodesics in $\Gamma$ are longer by a definite proportion. In particular, $\Gamma_t \in U_k\setminus U_{k+1}$ for all small enough $t\geq 0$. Let $\tau =\tau(\Gamma)$ be the supremum of times $t\in [0,\log(1/s)]$ such that $\Gamma_t \in U_k\setminus U_{k+1}$. Note that $\Gamma_\tau \in U_k$ since $U_k$ is closed and $\Gamma_t$ varies continuously. On the other hand, the union of the systoles in $\Gamma_\tau$ cannot be equal to a subgraph of rank exactly $k$ otherwise we could continue the deformation for $t > \tau$, so we have $\Gamma_\tau \in U_{k+1}$. This is unless $\tau =  \log(1/s)$, in which case the systoles in $\Gamma_\tau$ cover the whole graph and thus $\Gamma_\tau \in V_n' \subseteq V_n \subseteq U_{k+1}$ in that case too.

 The deformation retraction $U_k \times [0,1] \to U_{k}$ onto $U_{k+1}$ is defined by sending $(\Gamma, t)$ to $\Gamma_{t\cdot \tau(\Gamma)}$ if $\Gamma \in U_k\setminus U_{k+1}$ and to $\Gamma$ if $\Gamma \in U_{k+1}$. This map is clearly continuous, $\mathrm{Out}(F_n)$-equivariant, equal to the identity on $U_k$ at $t=0$ and on $U_{k+1}$ for all $t$, and a retract onto $U_{k+1}$ at $t=1$.
 
 The final deformation retraction of $V_n$ onto $V_n'$ can be defined similarly, by shrinking all the edges that do not belong to any systole (and expanding the rest to keep the volume constant) until either the complementary components have been shrunk to points or some new systole passing through a complementary component appears. Once again, the deformation retraction is performed in stages, ordered according to the number of edges that do not belong to any systole (recall that the metric graphs in $X_n$ are not allowed to have vertices of degree $1$ or $2$, so there are at most $3n-3$ edges).
\end{proof}

\section{Teichm\"uller space}

It remains to prove \thmref{thm:not_spine_teich} stating that the set $\calW_g$ of surfaces in $\calT_g$ whose systoles generate a finite-index subgroup in homology does not contain any equivariant spine. We simply explain how this follows from results in \cite{sublinear} and \cite{dim}.

\begin{proof}[Proof of \thmref{thm:not_spine_teich}]
Theorem 1.1 in \cite{sublinear} states that for every $\eps >0$, there exists some $g \geq 2$ and a closed hyperbolic surface $X$ such that the systoles in $X$ fill (so that $X \in \calV_g$) but there are fewer than $\eps g$ of them. By \cite[Proposition 5.1]{sublinear}, the surface constructed is such that $\Isom(X)$ acts transitively on a tiling of $X$ by copies of a quadrilateral $Q$ with three right angles and one angle of $\pi/q$ for some large integer $q$. Taking $H = \Isom(X)$ as a subgroup in the extended mapping class group $\Mod_g^{\pm}$, we see that the locus of fixed points $\Fix(H)$ in $\calT_g$ is $1$-dimensional because it is isometric to the Teichm\"uller space of the quotient orbifold $Q=X / \Isom(X)$. Trigonometric identities between the side lengths of a quadrilateral with three right angles \cite[p.454]{Buser} imply that this space is $1$-dimensional.

Now, varying the shape of $Q$ has the effect of changing the right-angled regular $2q$-gon $P$ used to construct $X$ into a semi-regular right-angled polygon with side lengths alternating between two values $t$ and $s(t)$ as in \cite[Section 2]{dim}. If we denote the deformed surface by $X_t$, then the arguments in \cite[Section 2]{dim} and \cite[Proposition 4.1]{sublinear} can be easily modified to show that the systoles in $X_t$ are either the red curves or the blue curves (or both) in the language of these papers. That is to say, the systoles in $X_t$ are a subset of those in $X$. In particular, there are fewer than $\eps g$ systoles in $X_t$ for every $X_t \in \Fix(H)$. Since any finite-index subgroup of $H_1(X,\ZZ)$ has rank $2g$, we obtain that $\Fix(H)$ is disjoint from $\calW_g$ as long as $\eps < 2$. By \lemref{lem:fix}, it follows that $\calW_g$ does not contain any $\Mod_g^{\pm}$-equivariant spine, for otherwise $\calW_g \cap \Fix(H)$ would be non-empty.

We then extend this statement to the mapping class group $\Mod_g$. Let $H^+ \leq H$ be the index-$2$ subgroup of orientation-preserving isometries. Then $Q^+ = X/H^+$ is an oriented orbifold without boundary that covers $Q$ with degree two, hence is equal to the double of $Q$ across its boundary, i.e., a sphere with $4$ cone points. The set $\Fix(H^+)$ is isomorphic to the Teichm\"uller space of $Q^+$, which is isometric to the hyperbolic plane, so $\Fix(H^+)$ is a Teichm\"uller disk $D$ containing the geodesic $L = \Fix(H)$.

Suppose that $S \subseteq \calW_g$ is a $\Mod_g$-equivariant spine for $\calT_g$. Since $\calW_g$ is disjoint from $L$, so is $S$. By \lemref{lem:fix}, $S\cap D$ is a deformation retract of $D$, so it is connected, hence contained in one of the two half planes bounded by $L$. On the other hand, $S \cap D$ is invariant under the action of the stabilizer $K$ of $D$ in $\Mod_g$. This stabilizer $K$ contains a copy of the pure mapping class group $\PMod(Q^+)$ since all homeomorphisms of $Q^+$ fixing the cone points lift to $X$. In turn, $\PMod(Q^+)$ acts on $\HH^2 \cong \calT(Q^+) \cong D$ as the principal congruence subgroup of level two $\Gamma(2)$, a finite-index subgroup in $\SL_2(\ZZ)$. In particular, $K$ is a lattice in $\Isom^+(D)$, hence its limit set is all of $\partial D$. This contradicts the previous observation that $S \cap D$ is $K$-invariant and contained in a half-plane.
\end{proof}

\begin{remark}
Similarly as for Outer space, the above argument shows that for any $\rho > 0$, the set of hyperbolic surfaces that have at least $\rho\, g$ systoles does not contain an equivariant spine for infinitely many $g$.
\end{remark}

\bibliographystyle{amsalpha}
\bibliography{biblio}

\end{document}